\documentclass[12pt,openany]{article}
\usepackage{amssymb, amsmath, latexsym, a4,amsthm,amsfonts}
\usepackage[latin1]{inputenc}
\usepackage{graphicx}
\usepackage[all]{xy}
\usepackage[usenames]{color}

\newtheorem{definition}{Definition}[section]
\newtheorem{theorem}[definition]{Theorem}
\newtheorem{proposition}[definition]{Proposition}
\newtheorem{lemma}[definition]{Lemma}
\newtheorem{corollary}[definition]{Corollary}
\newtheorem{remark}[definition]{Remark}
\newtheorem{example}[definition]{Example}

\newcommand{\ab}{\ensuremath{\mathsf{ab}}}

\newcommand{\Ker}{{\sf Ker}}

\newcommand{\Lie}{\ensuremath{\mathsf{Lie}}}

\newcommand{\Lieh}{\ensuremath{\mathfrak{h}}}
\newcommand{\Lieg}{\ensuremath{\mathfrak{g}}}
\newcommand{\Liet}{\ensuremath{\mathfrak{t}}}
\newcommand{\Lieq}{\ensuremath{\mathfrak{q}}}
\newcommand{\Liea}{\ensuremath{\mathfrak{a}}}
\newcommand{\Lieb}{\ensuremath{\mathfrak{b}}}
\newcommand{\Liem}{\ensuremath{\mathfrak{m}}}

\newcommand{\Lien}{\ensuremath{\mathfrak{n}}}
\newcommand{\Lief}{\ensuremath{\mathfrak{f}}}
\newcommand{\Lies}{\ensuremath{\mathfrak{s}}}
\newcommand{\Lier}{\ensuremath{\mathfrak{r}}}
\newcommand{\Liep}{\ensuremath{\mathfrak{p}}}

\newcommand{\Liej}{\ensuremath{\mathfrak{j}}}
\newcommand{\Lieu}{\ensuremath{\mathfrak{u}}}

\newcommand{\im}{{\sf Im}}
\newcommand{\Id}{{\sf id}}

\newbox\pullbackbox
\setbox\pullbackbox=\hbox{\xy 0;<1mm,0mm>: \POS(4,0)\ar@{-} (0,0) \ar@{-} (4,4)
\endxy}

\newdir{>>}{{}*!/3.5pt/:(1,-.2)@^{>}*!/3.5pt/:(1,+.2)@_{>}*!/7pt/:(1,-.2)@^{>}*!/7pt/:(1,+.2)@_{>}}
\newdir{ >>}{{}*!/8pt/@{|}*!/3.5pt/:(1,-.2)@^{>}*!/3.5pt/:(1,+.2)@_{>}}
\newdir{ |>}{{}*!/-3.5pt/@{|}*!/-8pt/:(1,-.2)@^{>}*!/-8pt/:(1,+.2)@_{>}}
\newdir{ >}{{}*!/-8pt/@{>}}
\newdir{>}{{}*:(1,-.2)@^{>}*:(1,+.2)@_{>}}
\newdir{<}{{}*:(1,+.2)@^{<}*:(1,-.2)@_{<}}

\begin{document}

\centerline{\bf Some properties of the Schur multiplier and}
\centerline{\bf stem covers of Leibniz crossed modules}

\bigskip
\centerline{\bf José Manuel Casas $^1$, Hajar Ravanbod $^{2}$ }
\bigskip \bigskip

\centerline{$^1$Dpto. Matemática Aplicada, Universidade de Vigo,  E. E. Forestal}
\centerline{Campus Universitario A Xunqueira, 36005 Pontevedra, Spain}
\centerline{ {E-mail address}: jmcasas@uvigo.es}
\bigskip

 \centerline{$^2$Faculty of Mathematical Sciences, Shahid Beheshti
 	University,}
 \centerline{ G. C., Tehran, Iran}
 \centerline{{E-mail address}: hajarravanbod@gmail.com}
\bigskip

\date{}

\bigskip \bigskip \bigskip

\noindent {\bf Abstract:}
In this article we investigate the interplay between stem covers, the Schur multiplier of Leibniz crossed modules and the non-abelian exterior product of Leibniz algebras. In concrete, we obtain a six-term exact sequence associated to a central extension of Leibniz crossed modules, which is useful to characterize stem covers. We show the existence of stem covers and determine the structure of all stem covers of Leibniz crossed modules. Also, we give the connection between the stem cover of a Lie crossed module in the categories of Lie and Leibniz crossed modules respectively.

\bigskip



\section{Introduction}
Leibniz algebras are algebraic structures introduced by  Bloh  in  $\cite{BL,BL2}$   as a non-skew symmetric generalization of Lie algebras. In the 90's Loday rediscovered and developed them \cite{Lo 1, Lo 2} when he handled periodicity phenomena in algebraic K-theory \cite{LP}. This  structure is not only important by algebraic reasons, but also for its applications in other branches such as Geometry  or Physics (see for instance \cite{DW, FLRO, HM, Lod}).

A {Leibniz algebra} is a $\mathbb{K}$-vector space ${\Lieq}$ equipped with a linear map  $[-,-]:{\Lieq} \times {\Lieq}\longrightarrow {\Lieq} $ satisfying the Leibniz identity
$[x, [y, z]] = [[x, y], z] - [[x, z], y]$, for all $x, y, z \in {\Lieq}$. If we assume $[x,x]=0$ for all $x \in \Lieq$, then $\Lieq$ is a Lie algebra.

 An active research line consists in the extension of properties from Lie algebras to Leibniz algebras. As an example of these generalizations, stem covers and stem extensions of a Leibniz algebras where studied in \cite{CL2};
in  \cite{Gnedbaye} was  extended to Leibniz algebras the notion of non-abelian tensor product of Lie algebras introduced by Ellis in \cite{Ellis};
 in \cite{DGMKh}, authors investigated the interplay between the non-abelian tensor and exterior product of Leibniz algebras with the low dimensional Leibniz homology of  Leibniz algebras.

Crossed modules of groups were described for the first time by Whitehead  in the late 1940s \cite{Wh} as an algebraic model for path-connected CW-spaces
whose homotopy groups are trivial in dimensions greater than 2. Crossed modules of different algebraic objects  can be regarded as algebraic structures that generalize simultaneously the notions of normal subobject and module. They were used in many branches of mathematics such as category theory, cohomology of algebraic structures, differential geometry or Physics. Also crossed modules were defined in different categories such as  Lie algebras, commutative algebras, etc. \cite{Po, KL}, either as tools or as algebraic structures in their own right.
Leibniz crossed modules were introduced in \cite{LP} to study the cohomology of Leibniz algebras. It was also used as coefficients for non-abelian (co)homology of Leibniz algebras in \cite{CKL}. Since Leibniz crossed modules  are generalizations of Lie crossed modules and Leibniz algebras, it is therefore of interest to extend results from  Leibniz algebra and Lie crossed modules to Leibniz crossed modules.

Accordingly, in this paper we show that Leibniz crossed modules constitutes a semi-abelian category with enough projective objects, hence the Baer invariant $\frac{(\Lieu, \Lier, \mu) \cap [(\Liem, \Lief, \mu),(\Liem, \Lief, \mu)]}{[(\Lieu, \Lier, \mu), (\Liem, \Lief, \mu)]}$ associated to the projective presentation $0 \longrightarrow (\Lieu, \Lier, \mu) \longrightarrow (\Liem, \Lief, \mu) \overset{(\pi_1,\pi_2)} \longrightarrow (\Lien, \Lieq, \delta) \longrightarrow 0$, called the Schur multiplier of the Leibniz crossed module $(\Lien, \Lieq, \delta)$ and denoted  by ${\cal M}(\Lien, \Lieq, \delta)$, plays a central role in the study of connections with the non-abelian exterior product of Leibniz algebras, in the study of stem covers of Leibniz crossed modules and in the study of connections with stem cover of Lie crossed modules.

The paper is organized as follows: in section \ref{prelim} we recall some basic categorical concepts such as the commutator of two ideals, the center, central extensions of  Leibniz crossed modules, etc. Moreover, we show the category of Leibniz crossed modules has enough projective objects.  In section \ref{multiplier} we describe the Schur multiplier of a Leibniz crossed module and analyze its interplay with the non-abelian exterior product of Leibniz algebras given in \cite{DGMKh}. In concrete we show that ${\cal M}(\Lien,\Lieq,\delta)  \cong  \Ker \left((\Lieq \curlywedge \Lien,\Lieq \curlywedge \Lieq, \Id \curlywedge \delta) \longrightarrow (\Lien,\Lieq,\delta) \right)$ and the existence of the six-term exact sequence
\[	
\xymatrix{
 & (I,\Lieb \curlywedge \Liep,\sigma \curlywedge \Id) \ar[r] & \mathcal{M}(\Lieh, \Liep, \sigma) \ar[r] & \mathcal{M}(\Lien, \Lieq, \delta)   \ar`r[dl] `[l]`[llld] `[d] [lld]\\
& (\Liea,\Lieb,\sigma) \ar[r] & (\Lieh,\Liep,\sigma)_{\rm ab} \ar[r] & (\Lien,\Lieq,\delta)_{\rm ab} \ar[r] & 0.
}
\]
associated to the  central extension of Leibniz crossed modules $(e): 0 \longrightarrow (\Liea,\Lieb,\sigma)$ $\longrightarrow (\Lieh,\Liep,\sigma)\stackrel{\varphi } \longrightarrow (\Lien,\Lieq,\delta) \longrightarrow 0$.

In section \ref{stem} we characterize and study properties of stem covers of Leibniz crossed modules through the Schur multiplier, in particular we show the existence of stem covers for an arbitrary Leibniz crossed module and similarly to a result of Schur in the group case \cite{Schur}, we determine the structure of all stem covers of a Leibniz crossed module whose Schur multiplier is finite dimensional. Finally, in last section, we study the connections between the  stem covers of a Lie crossed module in the categories of Lie and Leibniz crossed modules, respectively.

\section{Preliminaries on Leibniz crossed modules} \label{prelim}
This section is devoted to recall some basic definitions in the category of Leibniz
crossed modules, which
will be needed in the sequel.

\begin{definition} \cite{LP}
Let ${\Liem}$ and ${\Lien}$ be Leibniz algebras. A Leibniz action of ${\Liem}$ on $\Lien$ is a couple
of bilinear maps $\Liem \times \Lien \longrightarrow \Lien, (m,n)\longmapsto ^mn$,
and $\Lien \times \Liem \longrightarrow \Lien, (n,m) \longmapsto n^m$, satisfying
the following axioms:
\begin{center}
$^{[m,m']}n=^m{(^{m'}n)}+(^mn)^{m'}, \qquad ^m{[n,n']}=[^mn,n']-[^m{n'},n],$\\
$n^{[m,m']}={(n^m)}^{m'}-(n^{m'})^m, \qquad {[n,n']}^m= [n^m, n']+[n,{n'}^m],$\\
$^m{(^{m'}n)}=-^m{(n^{m'})}, \qquad [n,^m{n'}]=-[n,{n'}^m]$,
\end{center}
for each $m,m'\in \Liem, n, n'\in  \Lien$.
\end{definition}

\begin{definition}\cite{LP}
  A Leibniz crossed module $(\Lien,\Lieq, \delta)$ is a homomorphism of Leibniz algebras
$\delta : \Lien \longrightarrow \Lieq$ together with a Leibniz action of $\Lieq$ on $\Lien$ such that
\begin{enumerate}
\item[(i)] $\delta(^qn)=[q,\delta(n)],\qquad \delta(n^q)=[\delta(n),q],$
\item[(ii)] $^{\delta(n_1)}{n_2}=[n_1,n_2]={n_1}^{\delta(n_2)}$,
\end{enumerate}
for all $q \in \Lieq$ and $n_1,n_2 \in \Lien$.
\end{definition}

\begin{example}\label{2.1}\
\begin{enumerate}
\item[(i)] Let $\Lien$ be a two-sided ideal of a Leibniz algebra $\Lieq$, then $(\Lien,\Lieq,i)$ is a Leibniz crossed module, where $i$ is the inclusion map and the Leibniz action of $\Lieq$
on $\Lien$ is given by the Leibniz bracket. In this way, every Leibniz algebra $\Lieq$ can be regarded as crossed module in the two obvious ways $(0,\Lieq,i)$ and $(\Lieq,\Lieq, \Id )$.

\item[(ii)] For any $\Lieq$-module $\Liem$ the trivial map $0 : \Liem \to \Lieq$ is a crossed module with the trivial action of $\Lieq$ on the abelian Leibniz algebra $\Liem$.

\item[(iii)] Any homomorphism of Leibniz algebras $\delta : \Lien \to \Lieq$, with $\Lien$ abelian and $\im(\delta)$
in the center of $\Lieq$, provides a crossed module with $\Lieq$ acting trivially on $\Lien$.
\end{enumerate}
\end{example}

\begin{definition} \cite{CFCGMKh}
A homomorphism of Leibniz crossed modules,  $(\varphi,\psi):(\Lien,\Lieq,\delta)$ $\longrightarrow (\Lien',\Lieq',\delta')$ is a pair of Leibniz algebra homomorphisms $\varphi:\Lien
\longrightarrow \Lien'$ and $\psi:\Lieq \longrightarrow \Lieq'$ such that $\psi \circ \delta=\delta' \circ \varphi$ and $\varphi$ preserves the Leibniz action of $\Lieq$ via $\psi$, i.e. $\varphi(^qn)={^{\psi(q)}}{\varphi(n)}$
and $\varphi(n^q)={\varphi(n)}^{\psi(q)}$, for all $n\in \Lien$ and $q \in \Lieq$.

 A homomorphism  of crossed modules $(\varphi,\psi)$ is called {injective} if both $\varphi$ and $\psi$ are injective Leibniz algebra homomorphism. Also, $(\varphi,\psi)$ is called {surjective} if $\varphi$ and $\psi$ are onto maps.
\end{definition}

It is clear that Leibniz crossed modules constitute a category, which is denoted by $\mathbf{XLb}$. Theorem 10 in \cite{CKL}, in the particular case $n=2$,  provides the equivalence between the categories
$\mathbf{XLb}$ and ${\bf Cat}^1$-{\bf Lb} of ${\rm cat}^1$-Leibniz algebras (see also \cite{CFCGMKh}). Moreover, in \cite{BCDU} is showed that ${\bf Cat}^1$-{\bf Lb} is a modified category of interest,
which is a semi-abelian category. Hence $\mathbf{XLb}$ is a semi-abelian category (see also \cite{CFCGMKh}). Subobjects and normal subobjects in $\mathbf{XLb}$ are the crossed submodules and crossed ideals of a crossed module, that is, $(\Liem,\Lieg,\partial)$ is a crossed submodule of a crossed module $(\Lien,\Lieq,\delta)$ if $\Liem$ is a subalgebra of $\Lien$, $\Lieg$ is a subalgebra of $\Lieq$, $\partial = \delta_{\mid \Liem}$ and the Leibniz action of $\Lieg$ over $\Liem$ is the restriction of the Leibniz action of $\Lieq$ over $\Lien$. A crossed submodule $(\Liem,\Lieg,\partial)$ is said to be a crossed ideal of $(\Lien,\Lieq,\delta)$ when the action of $\Lieg$ over $\Lien$ belongs to $\Liem$ and the Leibniz action of $\Lieq$ over $\Liem$ belongs to $\Liem$.

 According to \cite{EVDL3} (see also \cite{EVDL1, EVDL2}), we have the following notions corresponding to the category $\mathbf{XLb}$:
\begin{itemize}
  \item  The \emph{commutator} of two crossed ideals $(\Lies, \Lieh, \delta)$ and $(\Liet, \Liej, \delta)$ of a crossed module of Leibniz algebras $(\Lien, \Lieq, \delta)$ is the ideal
$$[(\Lies, \Lieh, \delta), (\Liet, \Liej, \delta)] = \left( <D_{\Lieh}({\Liet}), D_{\Liej}({\Lies})>, [{\Lieh}, {\Liej}], \delta_{\mid} \right)$$
where $D_{\Lieh}({\Liet}) = \{ {^h}t, t^h \mid h \in {\Lieh}, t \in {\Liet} \}$ and $D_{\Liej}({\Lies}) = \{ {^j}s, s^j \mid j \in {\Liej}, s \in {\Lies} \}$.

  \item In particular, the \emph{derived crossed module} of a crossed module $(\Lien, \Lieq, \delta)$ is
$$(\Lien, \Lieq, \delta)' = [(\Lien, \Lieq, \delta), (\Lien, \Lieq, \delta)] = \left( D_{\Lieq}({\Lien}), [{\Lieq}, {\Lieq}], \delta_{\mid} \right)$$

\item Following \cite{CFCKhL}, the crossed ideal $(\Lien,\Lieq,\delta)^{\rm ann}$ is the crossed submodule $(\overline{\Lien},\Lieq_{\Lie},\bar{\delta})$, where $\overline{\Lien}$ is the ideal of $\Lien$
 generated by all elements $[n,n]$ and $[q,n]+[n,q]$, $q \in \Lieq,n \in \Lien$, and
$\Lieq^{\rm ann}$ is the ideal of $\Lieq$ generated by all elements $[q,q]$ for $q\in \Lieq$. Moreover $(\Lien,\Lieq,\delta)_{\Lie} = (\Lien,\Lieq,\delta) / (\Lien,\Lieq,\delta)^{\rm ann}$ is a crossed module of Lie algebras.

\item Following \cite{CFCGMKh}, the crossed ideal $Z(\Lien, \Lieq, \delta) = \left( {\Lien}^{\Lieq}, st_{\Lieq}({\Lien}) \cap Z({\Lieq}), \delta_{\mid} \right)$ is the \emph{center }of the crossed module $(\Lien,\Lieq,\delta)$, where $Z(\Lieq)$  is center of $\Lieq$, ${\Lien}^{\Lieq}  = \{ n \in {\Lien} \mid {^q}n = n^q =0, \text{for all}\ q \in q \}$ and $st_{\Lieq}({\Lien})=
  \{ q \in {\Lieq} \mid {^q}n = n^q =0, \text{for all}\ q \in \Lieq \}$.

  \item  An  extension of Leibniz crossed modules $(e) : 0 \to (\Liea,\Lieb,\sigma)$ $\to (\Lieh,\Liep,\sigma) \to (\Lien,\Lieq,\delta) \to 0$  is said to be central if $(\Liea,\Lieb,\sigma) \subseteq Z(\Lieh,\Liep,\sigma)$, equivalently $[(\Liea,\Lieb,\sigma),(\Lieh,\Liep,\sigma)]=0$.

    \item    A Leibniz crossed module  $(\Lien, \Lieq,\delta)$ is said to be {\bf finite dimensional} if the Leibniz algebras $\Lien$ and $\Lieq$ are both finite dimensional.

      \item A Leibniz crossed module $(\Lien, \Lieq,\delta)$ is  {\bf perfect} if it coincides with its commutator crossed submodule and it is {\bf abelian} if it coincides with its center. It is easy to show that $(\Lien, \Lieq,\delta)$ is abelian if and only if $\Lien$ and $\Lieq$ abelian Leibniz algebras and the Leibniz action of $\Lieq$ over $\Lien $ is trivial. We will denote the category of abelian crossed modules by {\bf AbXmod}. Obviously $(\Lien, \Lieq,\delta)_{\rm ab} = (\Lien, \Lieq,\delta)/[(\Lien, \Lieq,\delta),(\Lien, \Lieq,\delta)] = \left( \Lien/D_{\Lieq}(\Lien), \Lieq/[\Lieq, \Lieq],\overline{\delta} \right)$ is an abelian crossed module called the abelianization of $(\Lien, \Lieq,\delta)$.

\end{itemize}

\begin{theorem} \label{enough projectives}
$\mathbf{XLb}$ is a category with enough projective objects.
\end{theorem}
\begin{proof}
There is a faithful functor ${\bf U}_1 : {\bf XLb} \to {\bf Lb}$, which assigns to a  Leibniz crossed module $(\Lien, \Lieq, \delta)$ the direct product of Leibniz algebras
$\Lien \times \Lieq$. Now we define the functor ${\bf F}_1 : {\bf Lb} \to {\bf XLb}$ that assigns to any Leibniz algebra {\Lieh} the inclusion crossed module $(\bar{\Lieh},
{\Lieh} \ast {\Lieh}, inc)$, where $\ast$ is the coproduct of Leibniz algebras, with the natural inclusions  $i_1, i_2 : {\Lieh} \to {\Lieh} \ast {\Lieh}$, and $\bar{\Lieh}$ is
the kernel of the retraction $p_2 : {\Lieh} \ast {\Lieh} \to {\Lieh}$ determined by the conditions $p_2 \circ i_1 = 0$ and $p_2 \circ i_2 = \Id_{\Lieh}$. A direct adaptation of the proof of \cite[Proposition 2.1.1]{CIL} to Leibniz algebras case shows that ${\bf F_1}$ is left adjoint to ${\bf U_1}$.

Now consider the forgetful functor ${\bf U_2} : {\bf Lb} \to {\bf Set}$ that assigns to a Leibniz algebra {\Lieq} its underlying set. It is well-know (see for instance \cite{DGMKh}) that ${\bf U_2}$ has as left adjoint the free Leibniz algebra functor ${\bf F_2} : {\bf Set} \to {\bf Lb}$.

Hence the composition $\left( {\bf F}, {\bf U} \right) = \left( {\bf F}_1 \circ {\bf F_2}, {\bf U_2} \circ {\bf U_1} \right)$ is an adjoint pair, so the free crossed module of Leibniz algebras
${\bf F}(X)$, for $X \in {\bf Set}$, is a projective object with respect to regular epimorphisms in {\bf XLb} and  any crossed module of Leibniz algebras $(\Lien, \Lieq, \delta)$ admits o projective
presentation by means of the counit of the adjunction ${\bf F} \circ {\bf U} (\Lien, \Lieq,\delta) \twoheadrightarrow (\Lien, \Lieq, \delta)$.
\end{proof}

The following lemma is useful in our investigation.

\begin{lemma}\label{projective}\
\begin{enumerate}
\item[(i)]  An abelian crossed module $(A,B,\mu)$  is projective in the category {\bf AbXmod} if and only if $\mu$ is injective.
\item[(ii)]  Let $(\mathfrak{m},\Lief,\mu)$ be a projective Leibniz crossed module, then:
\begin{enumerate}
\item[(a)] any crossed submodule of $(\mathfrak{m},\Lief,\mu)_{\rm ab}$ is projective in  {\bf AbXmod}.
\item[(b)] the homomorphism $\mu$ is injective, $\Lief$ and $\Lief/ \mu(\Liem)$ are projective Leibniz algebras  and $HL_i(\Lief/\mu(\Liem))=0, i \geq 2$.
\end{enumerate}
 \end{enumerate}
\end{lemma}
\begin{proof}
{\it (i)} We can consider $(A,B,\mu)$ as a Leibniz crossed module. According to Theorem \ref{enough projectives} $(A,B,\mu)$ admits a projective presentation by means of the counit of the adjunction
${\bf FU}(A,B,\mu) \overset{(\pi_1, \pi_2)}\twoheadrightarrow (A,B,\mu)$. If $(A,B,\mu)$ is projective, then the morphism $(\pi_1, \pi_2)$ is split. Thus $(A,B,\mu)$ is isomorphic to a crossed submodule of ${\bf FU}(A,B,\mu) = (\overline{F(X)}, F(X) \ast F(X), inc)$, so $\mu$ is injective.

Conversely, let $(A,B,\mu)$ be aspherical (that is $\mu$ is injective), and $X$, $Y$ be the basis of $A$ and $B$, respectively.
We may assume that $X \subseteq Y$. Let $(\delta_1, \delta_2):(A,B,\mu)\longrightarrow (T_2,L_2,\sigma_2)$ and $(\varepsilon_1, \varepsilon_2) :(T_1,L_1,\sigma_1) \twoheadrightarrow (T_2,L_2,\sigma_2)$ be homomorphisms of crossed modules in {\bf AbXmod}, such that $(\varepsilon_1, \varepsilon_2)$ is surjective. There is a homomorphism $\theta_1:A\longrightarrow T_1$ such that $\varepsilon_1 \circ \theta_1=\delta_1$. We define the map $h:Y \longrightarrow L_1$, as $h(x)=\sigma_1 \circ \theta_1(x)$ if $x \in X$, otherwise $h(x)=l_x$, where $l_x$ is a preimage of $\delta_2(x)$ via $\varepsilon_2$, that is $\varepsilon_2(l_x) = \delta_2(x)$.
Then $h$ extends to a homomorphism $\theta_2 : B \longrightarrow L_2$. It is readily verified that $(\theta_1,\theta_2):(A,B,\mu)\longrightarrow (T_1,L_1,\sigma)$ is a homomorphism of crossed modules.

{\it (ii)}  {\it (a)} The abelianization functor ${\bf Ab} : {\bf XLb}\longrightarrow {\bf AbXmod},$ ${\bf Ab}(\Lien,\Lieq,\delta)= (\Lien,\Lieq,\delta)_{\rm ab}$ is left adjoint of the inclusion functor $inc : {\bf AbXmod} \longrightarrow {\bf XLb}$.  Since  the inclusion functor  preserves epimorphisms, then $\bf Ab$   preserves projective objects. Now, if $(\mathfrak{m},\Lief,\mu)$ is a projective Leibniz crossed module, then $(\mathfrak{m},\Lief,\mu)_{\rm ab}$ is also projective in the category  ${\bf AbXmod}$ and the result follows.

{\it (b)} The homomorphism $\mu$ is injective by applying an argument similar to the proof of statement {\it (i)}.
  Now,  assume  there exists the homomorphisms of Leibniz algebras $g:\Lieq\twoheadrightarrow \Lieq_2$ and $h:\Lief/\mu(\mathfrak{m})\longrightarrow \Lieq_2$. So, we have the induced morphisms of crossed modules
  $(0,g):(0,\Lieq_1,i)\twoheadrightarrow (0,\Lieq_2,i)$ and $(0,h \circ \pi):(\mathfrak{m},\Lief,\mu) \longrightarrow (0,\Lieq_2,i)$. By the assumptions, there is a morphism of crossed modules $(\beta_1,\beta_2):(\mathfrak{m}
  ,\Lief,\mu)\longrightarrow (0,\Lieq_1,i)$, such that $g \circ \beta_2 =h \circ \pi$. Since $\beta_2(\mu(\mathfrak{m}))=0$, $\bar{\beta}:\Lief/\mu(\mathfrak{m})\longrightarrow \Lieq_1$ is induced by $\beta_2$.
  It is easy to check that $g \circ \overline{\beta}=h$, so $\Lief/\mu(\mathfrak{m})$ is projective.  Similarly, $\Lief$ is projective.
  Finally, since $HL_{\ast}(\Lief/\mu(\mathfrak{m}))= TOR_{\ast}^{UL(\Lief/\mu(\mathfrak{m}))} \left(  U \left( (\Lief/\mu(\mathfrak{m}) \right)_{\Lie}) \right)$ \cite[Theorem 3.4]{LP}, and $\left( \Lief/\mu(\mathfrak{m})\right)_{\Lie}$ is a projective Lie algebra, hence $U\left ( \left( \Lief/\mu(\mathfrak{m} ) \right)_{\Lie} \right)$ is projective, then the result follows for any $i \geq 2$.
\end{proof}

\begin{remark}
 Let $\Lieq$ be a projective Leibniz algebra. Then Lemma \ref{projective} {\it (ii)} applied to the Leibniz crossed module $(0, \Lieq, i)$, implies that $HL_i(\Lieq)=0$ for $i\geq 2$.
\end{remark}


\section{Schur multiplier of Leibniz crossed modules} \label{multiplier}

Due to Theorem \ref{enough projectives}, any Leibniz crossed module has a projective presentation $0 \longrightarrow (\Lieu, \Lier, \mu) \longrightarrow (\Liem, \Lief, \mu) \overset{(\pi_1,\pi_2)} \longrightarrow (\Lien, \Lieq, \delta) \longrightarrow 0$ and following \cite[Theorem 6.9 and Corollary 6.10]{EVDL2} the quotient
$$\frac{(\Lieu, \Lier, \mu) \cap [(\Liem, \Lief, \mu),(\Liem, \Lief, \mu)]}{[(\Lieu, \Lier, \mu), (\Liem, \Lief, \mu)]}$$
is a Baer invariant, which means it doesn't depend on the chosen free presentation. By analogy with other algebraic theories, we call this term Schur multiplier of the Leibniz crossed module $(\Lien, \Lieq, \delta)$ and we denote it by ${\cal M}(\Lien, \Lieq, \delta)$. Moreover,  associated to a short exact sequence of Leibniz crossed modules  $(e) : 0 \longrightarrow (\Liea, \Lieb, \sigma)  \overset{(i_1,i_2)}\longrightarrow (\Lieh, \Liep, \sigma) \overset{(f_1,f_2)}\longrightarrow(\Lien, \Lieq, \delta) \longrightarrow 0$ there exists the
following five-term exact sequence:
\begin{equation} \label{five term}
 {\cal M}(\Lieh, \Liep, \sigma)  \to  {\cal M}(\Lien, \Lieq, \delta) \stackrel{\theta_{\ast}(e)} \to  \frac{(\Liea, \Lieb, \sigma)}{[(\Liea, \Lieb, \sigma), (\Lieh, \Liep, \sigma)]}
 \to  (\Lieh, \Liep, \sigma)_{\rm ab} \to (\Lien, \Lieq, \mu)_{\rm ab} \to 0.
\end{equation}


\subsection{Non-abelian tensor  and exterior product}

\begin{definition} \label{non abelian} \cite{Gnedbaye}
Let $\Liem$ and $\Lien$ be Leibniz algebras with mutual Leibniz actions on each other. The non-abelian
tensor product of $\Liem$ and $\Lien$, denoted by $\Liem \star \Lien$, is the Leibniz algebra
generated by the symbols $m \ast n$ and $n \ast m$, for all $m \in \Liem$ and $n \in \Lien$, subject to the
following relations:
\[  \footnotesize{ \begin{array}{ll}
(1a) ~ k(m\ast n)=km\ast n= m\ast kn, &(1b)~ k(n\ast m)=kn\ast m=n\ast km,\\
(2a) ~ (m+m')\ast n=m\ast n + m' \ast n, &(2b)~ (n+n')\ast m=n\ast m+n'\ast m,\\
(2c)~ m\ast(n+n')=m\ast n+m\ast n', &(2d) ~ n\ast (m+m')=n\ast m +n\ast m',\\
(3a)~ m\ast [n,n']=m^n\ast n' -m^{n'}\ast n,&(3b)~ n\ast[m,m']=n^m\ast m'-n^{m'}\ast m\\
(3c)~ [m,m']\ast n=^mn\ast m'-m\ast n^{m'},&(3d) ~ [n,n']\ast m={^n}m\ast n' - n\ast m^{n'},\\
(4a)~ m\ast ^{m'}n=-m\ast n^{m'},\qquad &(4b)~  n\ast ^{n'}m=-n\ast m^{n'},\\
(5a) ~ m^n\ast ^{m'}n'=[m\ast n,m'\ast n']=^mn\ast m'{^{n'}},&(5b) ~ {^n}m\ast {n'^{m'}}=[n\ast m,n'\ast m']=n^m \ast {^{n'}{m'}},\\
(5c) ~  m^n\ast {n'}{^{m'}}=[m\ast n,n'\ast m']=^mn\ast ^{n'}m', &(5d)~ {^n}m\ast ^{m'}n'=[n\ast m,m'\ast n']=n^m\ast m'^{n'},
\end{array}} \]
for all $k \in \mathbb{K},m,m' \in \Liem$ and $n,n' \in \Lien$.
\end{definition}

Let us consider two Leibniz crossed modules $\eta:\Liem \longrightarrow \Lieq$ and $\delta: \Lien \longrightarrow \Lieq$. Then there
are induced Leibniz actions of $\Liem$ and $\Lien$ on each other via the action of $\Lieq$. Therefore, we can
consider the non-abelian tensor product $\Liem\star \Lien$. In \cite{DGMKh} is defined $\Liem\square \Lien$ as the vector subspace
of $\Liem  \star\Lien$ generated by the elements $m \ast n'- n \ast m'$ such that $\eta(m) = \delta(n)$ and
$\eta(m' ) = \delta(n')$. The vector subspace $\Liem\square\Lien$ is contained in the center of $\Liem \star \Lien$, so in
particular it is an ideal of $\Liem \star \Lien$  \cite[Proposition 1]{DGMKh}.

\begin{definition} \cite{DGMKh} The non-abelian exterior product $\Liem \curlywedge \Lien$ of $\Liem$ and $\Lien$ is the quotient
 $$\Liem \curlywedge \Lien=\frac{\Liem\star \Lien}{\Liem\square \Lien}.$$
The cosets of $m \ast n$ and $n \ast m $ will be denoted by $m \curlywedge n$ and $n \curlywedge m$, respectively.
\end{definition}

Given a crossed module $(\Lien,\Lieq,\delta)$, by the Leibniz action of $\Lieq$ on $\Lien$ and the Leibniz action of $\Lien$ on $\Lieq$, given by $\delta$,
we can form the non-abelian  tensor product of $\Lieq \star\Lien$ and $\Lieq \star \Lieq$.
As explained in \cite[Proposition 4.3]{Gnedbaye}, the homomorphisms $\lambda_{\Lieq}:\Lieq \star\Lien\rightarrow \Lieq $, $\lambda_{\Lieq}(q\ast n)={^q}n, \lambda_{\Lieq}(n\ast q)={n}^q$ and $\mu_{\Lieq}:\Lieq \star \Lieq\rightarrow \Lieq$, $\mu_{\Lieq}(q\ast q')=[q,q']$ are Leibniz crossed modules,
where the Leibniz action of $\Lieq$ on $\Lieq\star \Lien$ is given by:
\[
\begin{array}{lcl}
^q(q'\ast n')=[q,q']\ast n'-{^q}{n'}\ast q',& \quad & ^q(n'\ast q')={^q}{n'}\ast q' -[q,q']\ast n', \\
(q'\ast n')^q=[q',q]\ast n'+q'\ast n'^{q}, & \quad & (n'\ast q')^q=n'^q\ast q'+ n' \ast [q',q].
\end{array}
\]
The  Leibniz  action of $\Lieq $ on $\Lieq\star \Lieq$ are defined similarly. It is apparent that $\lambda_{\Lieq}(\Lieq\square \Lien)=0$ and $\mu_{\Lieq}(\Lieq \square \Lieq)=0$, so the induced homomorphisms $\overline\lambda_{\Lieq}:\Lieq \curlywedge \Lien \rightarrow \Lieq$ and $\overline\mu_{\Lieq}:\Lieq \curlywedge \Lieq \rightarrow \Lieq$ are  Leibniz  crossed modules.

\begin{remark}
Given homomorphisms of Leibniz algebras $\varphi_1 : \Liem \longrightarrow \Lien$ and $\varphi_2 : \Liep \longrightarrow \Lieq$ such that $\Liem$ and $\Liep$, respectively $\Lien$ and $\Lieq$, have mutual Leibniz actions on each other, then there is induced a homomorphism $\varphi_1  \curlywedge \varphi_2 : \Liem \curlywedge \Liep \longrightarrow \Lien \curlywedge \Lieq$ defined by $\varphi_1  \curlywedge \varphi_2 (m \curlywedge p) = \varphi_1(m) \curlywedge \varphi_2(p), \varphi_1  \curlywedge \varphi_2(p \curlywedge m) = \varphi_2(p) \curlywedge \varphi_1(m)$.
\end{remark}

\begin{proposition} \label{central}
	Let $(\Lien,\Lieq,\delta)$ be a  Leibniz crossed module. Then,
\begin{enumerate}
	\item[(i)] There is a Leibniz action of $\Lieq \curlywedge \Lieq$ on $\Lieq \curlywedge \Lien$ defined by
	$^xy={^{\overline{\mu_ {\Lieq}}(x)}}y$ and $y^x=y^{\overline{\mu_ {\Lieq}}(x)}$, for all
	$x\in \Lieq\curlywedge \Lieq$ and $y\in \Lieq \curlywedge \Lien$.
	\item[(ii)] The map $\Id \curlywedge \delta:\Lieq \curlywedge \Lien \longrightarrow \Lieq \curlywedge \Lieq$
	with the  Leibniz action defined in statement {\it (i)} is a Leibniz crossed module.
	\item[(iii)]    There is a homomorphism $\phi = (\bar{\lambda}_{\Lien}, \bar{\mu}_{\Lieq}) : (\Lieq \curlywedge \Lien , \Lieq \curlywedge \Lieq , \Id \curlywedge \delta) \longrightarrow  (\Lien,\Lieq,\delta)$ such that $\Ker(\phi) \subseteq Z(\Lieq \curlywedge \Lien , \Lieq \curlywedge \Lieq , \Id \curlywedge \delta) $.
\end{enumerate}
\end{proposition}
\begin{proof}
	For statement {\it (i)}, thanks to the  Leibniz action of $\Lieq$ on $\Lieq \curlywedge \Lien$, everything can be
	easily checked.

For statement {\it (ii)}, it immediately follows, by using  relations (5a)--(5d) in Definition \ref{non abelian}, that
	$\Id \curlywedge \delta $ is a homomorphism of Leibniz algebras. Also, by using defining conditions of  Leibniz crossed module and  Leibniz action of $\Lieq$ on $\Lieq \curlywedge \Lien$, it is readily checked that $(\Id  \curlywedge \delta)(^xy)=[x,\Id  \curlywedge \delta(y)]$
	and $(\Id  \curlywedge \delta)(y^x)=[\Id  \curlywedge \delta(y),x]$ for all $x \in \Lieq \curlywedge \Lieq$, $y \in \Lieq \curlywedge \Lien$.
	
Now we indicate that $^{\Id  \curlywedge \delta(y_1)}y_2=[y_1,y_2]=y_1^{\Id  \curlywedge \delta(y_2)}$  for all $y_1,y_2 \in  \Lieq \curlywedge \Lien$.
	Let $y_i=q_i \curlywedge n_i$, for $i=1,2$, then we have
	\begin{alignat*}{1}
	^{q_1\curlywedge \delta(n_1)}(q_2\curlywedge n_2) & = ~ ^{[q_1,\delta(n_1)]}(q_2 \curlywedge n_2)\\
& =[[q_1,\delta(n_1)],q_2]\curlywedge n_2 -^{[q_1,\delta(n_1)]}n_2 \curlywedge q_2\\
	&=-[q_1,\delta(n_1)]\curlywedge n_2^{q_2} \\
	&=-q_1^{n_1}\curlywedge n_2^{q_2} \\
& =[q_1\curlywedge n_1,q_2\curlywedge n_2]\\
	& = {^{q_1}}{n_1} \curlywedge q_2^{n_2}\\
	&=[q_1,[q_2,\delta(n_2)]]\curlywedge n_1+q_1\curlywedge n_1^{[q_2,\delta(n_2)]}\\
	&=(q_1 \curlywedge n_1)^{q_2 \curlywedge\delta(n_2)}.
	\end{alignat*}
	For the other generators, it  can proved  in a similar way, so we obtain the result.

 For statement  {\it(iii)}, it is easy to check that $\phi$ is a crossed module homomorphism. To show that $\Ker(\phi) \subseteq Z(\Lieq \curlywedge \Lien ,\Lieq \curlywedge \Lieq, \Id \curlywedge \delta)$, let $x\in \Ker(\bar\lambda_n)$. We may assume that $x=q\curlywedge n$. So for any $q_1\curlywedge q_2\in \Lieq \curlywedge \Lieq$, we have:
	$$^{q_1\curlywedge q_2}{(q\curlywedge n)}=^{[q_1,q_2]}{(q\curlywedge n)}=[[q_1,q_2],q]\curlywedge n - ^{[q_1,q_2]}n\curlywedge q=[q_1,q_2]\curlywedge ^qn=0,$$
	by  relations (3c) and (4a) in Definition \ref{non abelian}.  With other generators can be proved  similarly, so we conclude that
	$\Ker(\bar\lambda_n) \subseteq {\Lieq \curlywedge \Lien}^{\Lieq \curlywedge \Lieq}$.

 Also, if $q_1\curlywedge q_2 \in \Ker(\bar \mu_q)$, then $[q_1,q_2]=0$. So it is an easy task to check that  $\Ker(\bar\mu_q) \subseteq Z(\Lieq \curlywedge \Lieq)\cap st_{\Lieq \curlywedge \Lieq}(\Lieq \curlywedge \Lien )$, and the result follows.
\end{proof}

	\begin{lemma} \label{3.1}
		Let $(\Lieh,\Liep,\sigma)$ be a Leibniz crossed module and $(\Liea,\Lieb,\sigma $ be a  crossed ideal of $(\Lieh,\Liep,\sigma)$ such that $(\Liea,\Lieb,\sigma) \subseteq Z(\Lieh,\Liep,\sigma)$. Then the map $$\sigma \curlywedge \Id : I \longrightarrow \Lieb \curlywedge \Liep$$ is an abelian Leibniz crossed module, where $I$ is the ideal of $\Liep \curlywedge \Lieh$ generated by all elements $p \curlywedge a, a \curlywedge p, b \curlywedge h$ and $h\curlywedge b$ for any $p \in \Liep, a \in \Liea, b\in \Lieb$ and $h\in \Lieh$.
	\end{lemma}
	\begin{proof}
		By the assumption $\Lieb \subseteq Z(\Liep)\cap st_{\Liep}(\Lieh)$ and $\Liea \subseteq {\Lieh}^{\Liep}$, so  by relation (5c) in Definition \ref{non abelian}, we have
\[
\begin{array}{lcl}
		[b\curlywedge p, p' \curlywedge b']=[b,p]\curlywedge [p',b']=0, & & [ a \curlywedge p,p' \curlywedge a' ]=a^p \curlywedge p'^{a'}=0, \\

		[b\curlywedge h,b'\curlywedge h']={^b}h \curlywedge {b'}^{h'}=0, & & [a \curlywedge p, p' \curlywedge h']=a^p \curlywedge p'^{h'}=0,
\end{array}
\]
		for all $b, b' \in \Lieb, p, p'\in \Liep$, $a, a' \in \Liea$, and $h,h' \in \Lieh$. Therefore, $I$ and $\Lieb \curlywedge \Liep$ are abelian Leibniz algebras. Evidently, the canonical homomorphism $\sigma \curlywedge \Id$ is an abelian Leibniz crossed module.
	\end{proof}

\begin{lemma}\label{3.2}
	Let $\varphi=(\varphi_1,\varphi_2):(\Lieh,\Liep,\sigma)\longrightarrow (\Lien,\Lieq,\delta)$ be a surjective homomorphism of Leibniz crossed modules. Then $\varphi \curlywedge \varphi = (\varphi_2 \curlywedge \varphi_1,\varphi_2 \curlywedge \varphi_2):(\Liep\curlywedge \Lieh,\Liep \curlywedge \Liep,\Id  \curlywedge \sigma)\longrightarrow (\Lieq\curlywedge \Lien,\Lieq \curlywedge \Lieq, \Id  \curlywedge \delta)$ is also a surjective homomorphism of Leibniz crossed modules.
\end{lemma}
\begin{proof}
	Obviously, homomorphism $(\varphi_1,\varphi_2)$ induces surjective  homomorphisms of Leibniz algebras $\varphi_2 \curlywedge \varphi_1: \Liep \curlywedge \Lieh\longrightarrow \Lieq\curlywedge \Lien$ and $\varphi_2\curlywedge \varphi_2:\Liep\curlywedge \Liep \longrightarrow \Lieq \curlywedge \Lieq$. It is easy to check that $(\varphi_2\curlywedge \varphi_1) \circ (\Id  \curlywedge \sigma)=(\Id  \curlywedge \delta) \circ (\varphi_2\curlywedge \varphi_2)$ and $\varphi_2 \curlywedge \varphi_1 $ preserves the action of crossed module via $\varphi_2 \curlywedge \varphi_2$, for instance
	\begin{alignat*}{1}
	\varphi_2 \curlywedge \varphi_1 \left(^{p_1\curlywedge p_2}{h\curlywedge p}\right) &=\varphi_2 \curlywedge \varphi_1\left( ^{[p_1,p_2]}h \curlywedge p \right)\\
&=\varphi_2 \curlywedge \varphi_1 \left(^{[p_1,p_2]}h \curlywedge p - [[p_1,p_2],p]\curlywedge h \right)\\
	&=^{\varphi_2([p_1,p_2])}{\varphi_1(h)} \curlywedge \varphi_2(p) -[[\varphi_2(p_1),\varphi_2(p_2)],\varphi_2(p)] \curlywedge \varphi_1(h)\\
	&=^{\varphi_2(p_1)\curlywedge \varphi_2(p_2)}{(\varphi_1(h)\curlywedge \varphi_2(p))}.
	\end{alignat*}
	Therefore, it is a homomorphism of crossed modules, as required.
\end{proof}

\begin{remark} \label{rem3.2}
By assumptions of  Lemma \ref{3.2}, let $\Ker(\varphi_1,\varphi_2)=(\Liea,\Lieb,\sigma)$ be, then we have the natural induced map of Leibniz algebras
$\psi_1:\Liep \curlywedge \Liea+\Lieb \curlywedge \Liem\longrightarrow \Liep\curlywedge \Liem$ and $\psi_2:\Liep \curlywedge \Lieb\longrightarrow \Liep\curlywedge \Liep$, such that $\im(\psi_1)=\Ker(\varphi_2\curlywedge \varphi_1)$ and $\im(\psi_2)=\Ker(\varphi_2\curlywedge \varphi_2)$ (see \cite{DGMKh}). So we may assume that the $\Ker(\varphi_2\curlywedge \varphi_1)$ is an ideal of $\Liep \curlywedge \Liem$ generated by all elements $p \curlywedge a, a \curlywedge p,b \curlywedge m$ and $m \curlywedge b$ for any $p\in \Liep, a \in \Liea,m \in \Liem$ and $b \in \Lieb$. Moreover, $\Ker(\varphi_2\curlywedge \varphi_2)$ is an ideal of $\Liep \curlywedge \Liep$ generated by $p\curlywedge b$ and $b \curlywedge p$ for all $p \in \Liep$ and $b \in \Lieb$.
\end{remark}

\subsection{Connections between the Schur multiplier and the non-abelian exterior product}

Following result shows  the connection between a free presentation of the given crossed module $(\Lien,\Lieq,\delta)$ and the non-abelian exterior product of Leibniz algebras $\Lieq$ and $\Lien$.

\begin{theorem}
	Let $0 \longrightarrow (\mathfrak{u},\Lier,\mu) \longrightarrow (\mathfrak{m},\Lief,\mu)\stackrel{(\pi_1,\pi_2)} \longrightarrow (\Lien,\Lieq,\delta) \longrightarrow 0$ be a free presentation of the Leibniz crossed module $(\Lien,\Lieq,\delta)$. Then there is an isomorphism
	$$(\Lieq\curlywedge \Lien,\Lieq \curlywedge \Lieq, \Id \curlywedge \delta) \cong (\frac{[\Lief,\mathfrak{m}]}{[\Lief,\mathfrak{u}]+[\Lier,\Liem]},\frac{[\Lief,\Lief]}{[\Lier,\Lief]},\bar{\mu}).$$
\end{theorem}
\begin{proof}
	According to Lemma \ref{projective} {\it (ii)}, $\Lief$ and $\Lief/\Liem$ are projective Leibniz algebras and so $HL_i(\Lief)=0=HL_i(\Lief/\Liem)$ for $i\geq 2$. So by \cite[Proposition 2 and Proposition 7]{DGMKh} the surjective homomorphism $\theta_{\Lief,\mathfrak{m}}: \Lief \curlywedge \mathfrak{m}\longrightarrow [\Lief,\mathfrak{m}]$ is an isomorphism. It is easy to see that $\theta_{\Lief,\mathfrak{m}} \left( \Ker(\pi_2\curlywedge \pi_1) \right) =[\Lief,\mathfrak{u}]+[\Lier,\mathfrak{m}]$ by Remark \ref{rem3.2}. So, it gives rise to the isomorphism
	$$\bar{\theta}_{\Lief,\mathfrak{m}}:\frac{\Lief \curlywedge \mathfrak{m}}{\Ker(\pi_2\curlywedge \pi_1)}\longrightarrow \frac{[\Lief,\mathfrak{m}]}{[\Lief,\mathfrak{u}]+[\Lier,\mathfrak{m}]}.$$
	Also, invoking  \cite[Theorem 4]{DGMKh}, $\Ker(\Lief\curlywedge \Lief\longrightarrow \Lief)=HL_2(f)=0$ so the surjection $\theta_{\Lief,\Lief}:\Lief\curlywedge \Lief\longrightarrow
	[\Lief, \Lief]$ is an isomorphism in which $\theta_{\Lief,\Lief}(\Ker(\pi_2\curlywedge \pi_2))=[\Lief,\Lier]$.
	So we obtain the induced isomorphism $\bar{\theta}_{\Lief,\Lief}:\Lief\curlywedge \Lief/\Ker(\pi_2\curlywedge \pi_2) \longrightarrow [\Lief,\Lief]/[\Lief,\Lier]$. Easily, the pair $(\bar{\theta}_{\Lief,\mathfrak{m}},\bar{\theta}_{\Lief,\Lief})$ is an isomorphism of crossed modules. Therefore, we conclude from Lemma \ref{3.2} that
	$$(\Lieq\curlywedge \Lien,\Lieq \curlywedge \Lieq, \Id \curlywedge \delta) \cong \frac{(\Lief \curlywedge \mathfrak{m} ,\Lief \curlywedge \Lief, \Id \curlywedge \mu)}{\Ker(\pi_2\curlywedge \pi_1,\pi_2\curlywedge \pi_2)} \cong \left( \frac{[\Lief,\mathfrak{m}]}{[\Lief,\mathfrak{u}]+[\Lier,\Liem]},\frac{[\Lief,\Lief]}{[\Lier,\Lief]},\bar{\mu} \right).$$
	The proof is complete.
\end{proof}

For any Leibniz algebra $\Lieq$ we have $HL_2(\Lieq)\cong \Ker(\Lieq \curlywedge \Lieq\longrightarrow \Lieq)$ \cite[Theorem 4]{DGMKh}.  As an immediate consequence of the above Theorem,  we generalize this result for Leibniz crossed modules as follows:

\begin{corollary}\label{cor3.1}
	Let $(\Lien,\Lieq,\delta)$ be a Leibniz crossed module. Then we have
\[
\begin{array}{rcl}
		{\cal M}(\Lien,\Lieq,\delta) & \cong & \Ker \left((\Lieq \curlywedge \Lien,\Lieq \curlywedge \Lieq, \Id \curlywedge \delta) \longrightarrow (\Lien,\Lieq,\delta) \right)\\
		& = & \left (\Ker(\Lieq \curlywedge \Lien \longrightarrow \Lien), \Ker(\Lieq \curlywedge \Lieq \longrightarrow \Lieq),\Id \curlywedge \delta \right).
	\end{array}
\]
\end{corollary}

\begin{remark}
Corollary \ref{cor3.1} shows that for any abelian Leibniz crossed module $(\Liea,\Lieb,\sigma)$ we have ${\cal M}(\Liea,\Lieb,\sigma)=(\Lieb\curlywedge \Liea,\Lieb\curlywedge \Lieb,\Id \curlywedge \sigma)$.
\end{remark}

Now we extend  sequence (\ref{five term}) to a six-term natural exact sequence as follows:

\begin{theorem}\label{th3.2}
	Let $(e): 0 \longrightarrow (\Liea,\Lieb,\sigma) \longrightarrow (\Lieh,\Liep,\sigma)\stackrel{\varphi } \longrightarrow (\Lien,\Lieq,\delta) \longrightarrow 0$ be a central extension of Leibniz crossed modules, then the exact sequence (\ref{five term}) can be extended one term further to the following natural exact sequence
\begin{equation} \label{six term}	
\xymatrix{
 & (I,\Lieb \curlywedge \Liep,\sigma \curlywedge \Id) \ar[r] & \mathcal{M}(\Lieh, \Liep, \sigma) \ar[r] & \mathcal{M}(\Lien, \Lieq, \delta)   \ar`r[dl] `[l]`[llld] `[d] [lld]\\
& (\Liea,\Lieb,\sigma) \ar[r] & (\Lieh,\Liep,\sigma)_{\rm ab} \ar[r] & (\Lien,\Lieq,\delta)_{\rm ab} \ar[r] & 0.
}
\end{equation}
\end{theorem}
\begin{proof}
	Considering the inclusion map $\beta:I\longrightarrow \Liep \curlywedge \Lieh$ and the functional homomorphism
		$\alpha: \Lieb \curlywedge \Liep \longrightarrow \Liep \curlywedge \Liep$ one easily sees that $\phi=(\beta,\alpha):(I,\Lieb \curlywedge \Liep,\sigma \curlywedge \Id )
		\longrightarrow (\Liep \curlywedge \Lieh,\Liep\curlywedge \Liep, \Id\curlywedge \sigma)$ is a homomorphism of crossed modules. Thanks to Remark \ref{rem3.2} we have the following commutative diagram:
	\[  \footnotesize{ \xymatrix{
& (I,\Lieb \curlywedge \Liep,\sigma \curlywedge \Id )  \ar[r]^{\phi} \ar[d]^{(\bar{\lambda}_{\Lieh}~,~ \bar{\lambda}_{\Liep}~)_{\mid}}& (\Liep\curlywedge \Lieh,\Liep\curlywedge \Liep, \Id \curlywedge\sigma) \ar[r]^{\varphi \curlywedge \varphi} \ar[d]^{(\bar{\lambda}_{\frak{h}}~,~ \bar{\lambda}_{\frak{p}}~)} & (\Lieq\curlywedge\Lien,\Lieq\curlywedge\Lieq,\Id \curlywedge\delta) \ar[r] \ar[d]^{(\bar{\lambda}_{\Lien}~,~ \bar{\lambda}_{\Lieq}~)} & 0\\
  0 \ar[r] & (\Liea,\Lieb,\sigma) \ar[r] &(\Lieh,\Liep,\sigma) \ar[r]^{\varphi} & (\Lien,\Lieq,\delta) \ar[r] & 0. }
} \]
Now the Snake Lemma (which is valid in any semi-abelian category \cite{BB}) completes the proof, thanks to Corollary \ref{cor3.1} and since $\im(\bar{\lambda}_{\frak{h}},\bar{\lambda}_{\frak{p}})_{\mid} = 0$, and  ${\sf Coker}(\bar{\lambda}_{\frak{h}},\bar{\lambda}_{\frak{p}}) \cong (\Lieh,\Liep,\sigma)_{\ab}$, and ${\sf Coker}(\bar{\lambda}_{\frak{n}},\bar{\lambda}_{\frak{q}}) \cong (\Lien,\Lieq,\delta)_{\ab}$.
\end{proof}

\begin{remark}
	If we consider Leibniz algebras as crossed modules in any of the two usual ways  (Example \ref{2.1} {\it (i)}), we get the corresponding results for Leibniz algebras in \cite[Proposition 7]{DGMKh}.
\end{remark}


 \section{Stem extensions and stem covers of Leibniz crossed modules} \label{stem}
 This section is devoted to generalize the notions of stem extension and stem cover of Leibniz algebras to Leibniz crossed modules context. For that the homomorphism $\theta_{\ast}(e)$ in exact sequence (\ref{five term}) plays a central role. When a Leibniz algebra is regarded as a Leibniz crossed module in the two usual ways (Example \ref{2.1} {\it (i)}), then the subsequent results recover the corresponding ones for stem extensions and stem cover of Leibniz algebras in \cite{CL2, EV}.

 \begin{definition}
 A central extension of Leibniz crossed modules $(e) : 0 \to (\Liea,\Lieb,\sigma)$ $\to (\Lieh,\Liep,\sigma) \stackrel{\varphi=(\varphi_1,\varphi_2)}\to (\Lien,\Lieq,\delta) \to 0$  is said to be a stem extension  if $(\Liea,\Lieb,\sigma) \subseteq [(\Lieh,\Liep,\sigma), $ $(\Lieh,\Liep,\sigma)]$.

 Also, if $(\Liea,\Lieb,\sigma)\cong {\cal M}(\Lien,\Lieq,\delta)$, then the stem extension $(e)$ is called a stem cover or covering of $(\Lien,\Lieq,\delta)$.
\end{definition}

The following result provides a characterization of stem extensions and stem covers.

\begin{proposition} \label{4.1}
  Let $(e) : 0 \to (\Liea,\Lieb,\sigma)$ $\to (\Lieh,\Liep,\sigma)\stackrel{\varphi=(\varphi_1,\varphi_2)}\to (\Lien,\Lieq,\delta) \to 0$  be a central extension of Leibniz crossed modules, then:
\begin{enumerate}
  \item[(i)] The following statements are equivalent:
  \begin{enumerate}
  \item[(a)] $(e)$ is stem extension of $(\Lien,\Lieq,\delta)$.
  \item[(b)] The homomorphism $\theta_{\ast}(e): {\cal M}(\Lien,\Lieq,\delta)\longrightarrow (\Liea,\Lieb,\sigma)$ is surjective.
  \item[(c)] The homomorphism $(\Liea,\Lieb,\sigma)\longrightarrow (\Lieh,\Liep,\sigma)/[(\Lieh,\Liep,\sigma),(\Lieh,\Liep,\sigma)]$ is the zero map.
  \item[(d)] The homomorphism $\frac{(\Lieh,\Liep,\sigma)}{[(\Lieh,\Liep,\sigma),(\Lieh,\Liep,\sigma)]} \longrightarrow \frac{(\Lien,\Lieq,\delta)}{[(\Lien,\Lieq,\delta),(\Lien,\Lieq,\delta)]}$ is an isomorphism.
  \end{enumerate}
 \item[(ii)] Under the assumption  ${\cal M}(\Lien,\Lieq,\delta)$ is finite dimensional in the central extension (e), the following statements are equivalent:
 \begin{enumerate}
  \item[(a)]  $(e)$ is a stem cover.
   \item[(b)]  $\theta_{\ast}(e)$ is an isomorphism.
    \item[(c)] The homomorphism $\frac{(\Lieh,\Liep,\sigma)}{[(\Lieh,\Liep,\sigma),(\Lieh,\Liep,\sigma)]} \longrightarrow \frac{(\Lien,\Lieq,\delta)}{[(\Lien,\Lieq,\delta),(\Lien,\Lieq,\delta)]}$ is an isomorphism and the induced homomorphism ${\cal M}(\Lieh,\Liep,\sigma) \longrightarrow {\cal M}(\Lien,\Lieq,\delta)$ is the zero map.
    \end{enumerate}
  \end{enumerate}
\end{proposition}
\begin{proof}
  Direct checking from exact sequence (\ref{five term}).
\end{proof}

\begin{corollary} \label{cor4.1}
  Let $(\Lien,\Lieq,\delta)$ be a perfect Leibniz crossed module. Then the central extension $(e) : 0 \to (\Liea,\Lieb,\sigma)$ $\to (\Lieh,\Liep,\sigma) \to (\Lien,\Lieq,\delta) \to 0$ is a stem cover if and only if $(\Lieh,\Liep,\sigma)_{\rm ab} = {\cal M}(\Lieh,\Liep,\sigma)=0$.
\end{corollary}
\begin{proof}
  According to Proposition \ref{4.1} {\it (i) (d)}, $(\Lieh,\Liep,\sigma)$ is a perfect crossed module. Hence $(\Lieh,\Liep,\sigma)_{\rm ab} = 0$.
  So we have $\Liep=[\Liep,\Liep]$ and $\Lieh=D_{\Liep}(\Lieh)$.

  We claim that the crossed module $(I,\Lieb \curlywedge \Liep,\sigma \curlywedge \Id)$ is trivial. Indeed,  thanks to \cite[Proposition 4.2]{Gnedbaye}, $\Lieb \star \Liep =\Lieb \otimes \Liep_{ab} \oplus \Liep_{ab}\otimes\Lieb=0$ and so $\Lieb \curlywedge \Liep=0$. Now, let $b \curlywedge h \in I$, then we can assume $h={^{p_0}}h_0$, for some $h_0 \in \Lieh, p_0\in \Liep$, then we have
  $$b\curlywedge h=b \curlywedge {^{p_0}}{h_0}=-b \curlywedge h_0^{p_0}=[b,p_0]\curlywedge h_0 -{^b}h_0 \curlywedge p_0=0,$$
 	by relations (4a) and (3c) in Definition \ref{3.1}. Similar computations can be done with the other generators of $I$. Thus, we can conclude that $I$ is trivial.
 	Then ${\cal M}(\Lieh,\Liep,\sigma)=0$ from sequence (\ref{six term}) and Theorem \ref{4.1} {\it (ii) (c)}.

 The converse is immediately followed from sequence (\ref{six term}) and Theorem \ref{4.1}.
\end{proof}

 The following proposition plays a basic role in the proofs of most of the subsequent results.
 \begin{proposition} \label{prop 4.2}
   Let $0 \longrightarrow (\mathfrak{u},\Lier,\mu) \longrightarrow (\mathfrak{m},\Lief,\mu) \stackrel{\pi=(\pi_1,\pi_2)}\longrightarrow (\Lien,\Lieq,\delta) \longrightarrow 0$ be a projective presentation of $(\Lien,\Lieq,\delta)$, then
   \begin{enumerate}
   \item[(i)] The following exact sequence is split
   $$0 \longrightarrow {\cal M}(\Lien,\Lieq,\delta) \longrightarrow (\bar{\mathfrak{u}},\bar{\Lier},\bar{\mu})\longrightarrow \frac{(\mathfrak{u},\Lier,\mu)}{(\mathfrak{u},\Lier,\mu)\cap [(\mathfrak{m},\Lief,\mu),(\mathfrak{m},\Lief,\mu)]} \longrightarrow 0,$$
   where $(\bar{\mathfrak{u}},\bar{\Lier},\bar{\mu}) = \frac{(\mathfrak{u},\Lier,\mu)}{[(\mathfrak{u},\Lier,\mu), (\mathfrak{m},\Lief,\mu)]}$.
   \item[(ii)]If $0 \longrightarrow (\Liea,\Lieb,\sigma) \longrightarrow (\Lieh,\Liep,\sigma) \stackrel {\gamma=(\gamma_1,\gamma_2)} \longrightarrow (\Lien_1,\Lieq_1,\delta_1) \longrightarrow 0$ is a stem extension of another Leibniz crossed module $(\Lien_1,\Lieq_1,\delta_1)$
     and $\alpha =(\alpha_1, \alpha_2) : (\Lien, \Lieq, \delta)\longrightarrow (\Lien_1,\Lieq_1, \delta_1)$ is a homomorphism of Leibniz crossed modules, then there exists a homomorphism $\overline{\beta} = (\overline{\beta_1}, \overline{\beta_2}) : (\bar{\mathfrak{m}},\bar{\Lief},\bar\mu)\longrightarrow
(\Lieh,\Liep,\sigma)$ such that $\bar\beta({\cal M}(\Lien,\Lieq,\delta)) \subseteq \bar\beta(\bar{\mathfrak{u}},\bar{\Lier},\bar{\mu})\subseteq (\Liea,\Lieb,\sigma)$, and the following diagram is commutative
 \[ \xymatrix{
0 \ar[r] & (\overline{\mathfrak{u}},\overline{\Lier},\overline{\mu}) \ar[r] \ar[d]& (\overline{\mathfrak{m}},\overline{\Lief},\overline{\mu}) \ar[r]^{\overline{\pi}} \ar[d]^{\overline{\beta}} & (\Lien,\Lieq,\delta) \ar[r] \ar[d]^{\alpha} & 0\\
0 \ar[r] & (\Liea,\Lieb,\sigma) \ar[r] &(\Lieh,\Liep,\sigma) \ar[r]^{\gamma} & (\Lien_1,\Lieq_1,\delta_1) \ar[r] & 0.
 } \]
 where $(\overline{\mathfrak{m}},\overline{\Lief},\overline{\mu})=\frac{(\Liem,\Lief,\mu)}{[(\mathfrak{u},\Lier,\mu), (\mathfrak{m},\Lief,\mu)]}$. Furthermore, if $\alpha$ is surjective, then so is $\bar{\beta}$, and $\bar{\beta} ({\cal M}(\Lien,\Lieq, \delta)) = (\Liea,\Lieb,\sigma)$.
  \end{enumerate}
 \end{proposition}
 \begin{proof}
{\it (i)} By the Isomorphism Theorem.

{\it (ii)} It is a straightforward adaptation of the proof of Lemma 3.3 in \cite{MSS}.
 \end{proof}

In the following, we determine the structure of stem covers of Leibniz crossed module which is analogous to the similar works in group and in Lie crossed modules \cite{RS,MSS}.

\begin{theorem} \label{Th 4.1}
  Let  $0 \longrightarrow (\mathfrak{u},\Lier,\mu) \longrightarrow (\mathfrak{m},\Lief,\mu) \longrightarrow (\Lien,\Lieq,\delta) \longrightarrow 0$ be a free presentation of a Leibniz crossed module $(\Lien,\Lieq,\delta)$. Then
  \begin{enumerate}
  \item[(i)] If $(\bar{\mathfrak{u}},\bar{\Lier},\bar{\mu}) \cong {\cal M}(\Lien,\Lieq,\delta)\oplus (\bar{\Liet},\bar{\Lies},\bar{\mu})$ for some  crossed ideal $\mathfrak{(t,s,\mu)}$ of $\mathfrak{(m,f,\mu)}$, where $(\bar{\Liet},\bar{\Lies},\bar{\mu})$ denotes the quotient $\frac{(\Liet, \Lies, \mu)}{[(\Lieu,\Lier,\mu),(\Liem, \Lief, \mu)]}$, then the extension $(e): 0 \longrightarrow \mathfrak{(u/t,r/s,\bar{\mu}) \longrightarrow
   (m/t,f/s,\bar{\mu}) \longrightarrow (n,q,\delta)} \longrightarrow 0$ is a stem cover of $\mathfrak{(n,q,\delta)}$.
  \item[(ii)] If ${\cal M}(\mathfrak{n,q,\delta})$ is finite dimensional and $(e_1): 0 \longrightarrow \mathfrak{(a,b,\sigma)
  \longrightarrow (h,p,\sigma) \longrightarrow }$ $\mathfrak{ (n,q,\delta)} \longrightarrow 0$ is a stem cover of $\mathfrak{(n,q,\delta)}$,
   then there is a  crossed ideal $\mathfrak{(t,s,\mu)}$ of $\mathfrak{(m,f,\mu)}$  satisfying statement {\it (i)} and such that $\mathfrak{(h,p,\sigma)} \cong \mathfrak{(m/t,f/s,\bar{\mu})}$  and $\mathfrak{(a,b,\sigma)} \cong \mathfrak{(u/t,r/s,\bar{\mu})}$.
   \end{enumerate}
\end{theorem}

\begin{proof}
{\it (i)} By the assumption, we have $\mathfrak{(u/t, r/s,\bar{\mu})}\cong {\cal M}(\Lien,\Lieq,\delta)$ and $(\mathfrak{u},\Lier,\mu)=(\mathfrak{u},\Lier,\mu)\cap [(\mathfrak{m},\Lief,\mu),(\mathfrak{m},\Lief,\mu)]+(\Liet,\Lies,\mu)\subseteq [(\mathfrak{m},\Lief,\mu),(\mathfrak{m},\Lief,\mu)]+(\Liet,\Lies,\mu)$, then
\[ \begin{array}{ll}
\mathfrak{(u/t,r/s,\bar{\mu})} &
\subseteq \frac{[(\mathfrak{m},\Lief,\mu),(\mathfrak{m},\Lief,\mu)]+(\Liet,\Lies,\mu)}{(\Liet,\Lies,\mu)}\cap Z\left( \frac{(\mathfrak{m},\Lief,\mu)}{(\mathfrak{t},\Lies,\mu)} \right)\\
&\subseteq[(\mathfrak{m/t,f/s},\bar{\mu}), (\mathfrak{m/t,f/s},\bar{\mu})]\cap Z \left( \frac{(\mathfrak{m},\Lief,\mu)}{(\mathfrak{t},\Lies,\mu)} \right),
\end{array} \]
so $(e)$ is stem cover of $(\Lien,\Lieq,\delta)$.

{\it (ii)} According to Proposition \ref{prop 4.2} {\it (ii)}, there is a surjective homomorphism $\bar{\beta} = (\bar{\beta}_1,\bar{\beta}_2) :
 (\bar{\mathfrak{m}},\bar{f},\bar{\mu})\longrightarrow ({\Lieh}, \Liep,\sigma)$ such that $\bar{\beta}({\cal M}\mathfrak{(n,q,\delta)})=\bar{\beta}(\mathfrak{\bar{u},\bar{r},\bar{\mu}})=
  \mathfrak{(a,b,\sigma)}$. Setting $\Ker( \bar{\beta})= \mathfrak{(\bar{t},\bar{s},\bar{\mu})}$,
   we have $\mathfrak{(h,p,\sigma)\cong (m,f,\mu)/(t,s,\mu)}$ and $\mathfrak{(a,b,\sigma)} \cong$
$\mathfrak{(u,r,\mu)/(t,s,\mu)}$. Also the restriction of $\bar{\beta}_{\mid}$ from
     ${\cal M}(\mathfrak{n,q,\delta})$ to $\mathfrak{(a,b,\sigma)}$ is surjective and since
     ${\cal M}(\mathfrak{n,q,\delta})$ is finite dimensional, so $\bar{\beta}_{\mid}$ is an isomorphism,
      therefore ${\cal M}(\mathfrak{n,q,\delta})\cap \Ker (\bar{\beta}) = \Ker (\bar{\beta}_{\mid})=0$.
   As the kernel of the restriction of $\bar{\beta}$ to $\mathfrak{(\bar{u},\bar{r},\bar{\mu})}$
    is $\Ker (\bar{\beta})$ and the image of this restriction is $\mathfrak{(a,b,\sigma)}$, so  the result follows.
\end{proof}

Thanks to above Theorem \ref{Th 4.1} and Proposition \ref{prop 4.2}, we provide the following important consequence.

\begin{corollary}
  Any Leibniz crossed module $(\Lien,\Lieq,\delta)$ admits at least one stem cover.

   In particular any Leibniz algebra admits at least one stem cover.
\end{corollary}
\begin{proof}
Let  $0 \longrightarrow (\mathfrak{u},\Lier,\mu) \longrightarrow (\mathfrak{m},\Lief,\mu) \longrightarrow (\Lien,\Lieq,\delta) \longrightarrow 0$ be a free presentation of a Leibniz crossed module $(\Lien,\Lieq,\delta)$. Then by Proposition \ref{prop 4.2} {\it (i)}, there is a crossed ideal $( \Liet,\Lies,\mu)$ of $(\Liem,\Lief,\mu)$ such that
	and
	 $(\bar{\mathfrak{u}},\bar{\Lier},\bar{\mu}) \cong {\cal M}(\Lien,\Lieq,\delta)\oplus (\bar{\Liet},\bar{\Lies},\bar{\mu})$.
	 Now the result follows by Theorem \ref{Th 4.1} {\it (i)}.
\end{proof}

It is  very interesting to find the relations between two stem cover of given Leibniz crossed module. In the following we prove that some crossed submodules and factor crossed modules of covering crossed modules are always isomorphic.

\begin{corollary}\label{iso}
  Let $(\Lien, \Lieq,\delta)$ be a Leibniz crossed module with finite dimensional Schur multiplier and let $(e_i): 0 \longrightarrow  {(\Liea_i,\Lieb_i,\sigma_i) \longrightarrow  (\Lieh_i,\Liep_i,\sigma_i) \stackrel{\varphi_i=(\varphi_{i1},\varphi_{i2})}\longrightarrow  (\Lien,\Lieq,\delta)} \longrightarrow  0$
  be two stem covers of $(\Lien,\Lieq,\delta)$, for $i=1,2$. Then
  \begin{enumerate}
  \item[(i)] $[(\Lieh_1,\Liep_1,\sigma_1), (\Lieh_1,\Liep_1,\sigma_1)]\cong [(\Lieh_2,\Liep_2,\sigma_2), (\Lieh_2,\Liep_2,\sigma_2)]$.
  \item[(ii)] $(\Lieh_1,\Liep_1,\sigma_1)/Z(\Lieh_1,\Liep_1,\sigma_1)\cong (\Lieh_2,\Liep_2,\sigma_2)/Z(\Lieh_2,\Liep_2,\sigma_2)$.
  \item[(iii)] $Z(\Lieh_1,\Liep_1,\sigma_1)/(\Liea_1,\Lieb_1,\sigma_1)\cong Z(\Lieh_2,\Liep_2,\sigma_2)/(\Liea_2,\Lieb_2,\sigma_2)$.
    \end{enumerate}
\end{corollary}
\begin{proof}
	{\it (i)} 	Let $(f):0 \longrightarrow (\mathfrak{u},\Lier,\mu) \longrightarrow (\mathfrak{m},\Lief,\mu) \stackrel{\pi=(\pi_1,\pi_2)} \longrightarrow (\Lien,\Lieq,\delta) \longrightarrow 0$ be a free presentation of
	$(\Lien,\Lieq,\delta)$. By applying a  similar argument to the proof  of  \cite[Theorem 3.7]{MSS}, we can show that the crossed modules $$[(\Lieh_i,\Liep_i,\sigma_i),(\Lieh_i,\Liep_i,\sigma_i)],\ (\Lieh_i,\Liep_i,\sigma_i)/Z(\Lieh_i,\Liep_i,\sigma_i), \ \text{and}\ Z(\Lieh_i,\Liep_i,\sigma_i)/(\Liea_i,\Lieb_i,\sigma_i),$$
are uniquely  determined by the free presentation $(f)$.

By virtue of Proposition \ref{prop 4.2} {\it (ii)} and Theorem \ref{Th 4.1} {\it (ii)}, there is a surjective homomorphism $\overline{\beta}: (\bar{\mathfrak{m}},\bar{\Lief},\bar\mu)\longrightarrow
	(\Lieh_1,\Liep_1,\sigma_1)$ such that $\bar{\beta} \left( {\cal M}(\Lien,\Lieq,\delta) \right) = (\Liea_1, \Lieb_1, \sigma_1)$. Since ${\cal M}(\Lien,\Lieq,\delta)$ is finite dimensional, then the restriction of $\bar{\beta}$ from ${\cal M}(\Lien,\Lieq,\delta)$ onto $(\Liea_1, \Lieb_1, \sigma_1)$ is an isomorphism. Thus, $0 = \Ker(\bar{\beta}_{\mid}) = \Ker(\bar{\beta}) \cap {\cal M}(\Lien,\Lieq,\delta)$, and it implies that $\Ker(\bar\beta) \cap [(\bar{\mathfrak{m}},\bar{\Lief},\bar\mu), (\bar{\mathfrak{m}},\bar{\Lief},\bar\mu)]=0$.  It is easy to see that $\bar \beta$ induces the surjective homomorphism $\hat{\beta}:[(\bar{\mathfrak{m}},\bar{\Lief},\bar\mu), (\bar{\mathfrak{m}},\bar{\Lief},\bar\mu)] \longrightarrow [(\Lieh_1,\Liep_1,\sigma_1),(\Lieh_1,\Liep_1,\sigma_1)]$ with $\Ker(\hat{\beta})=\Ker(\bar\beta) \cap [(\bar{\mathfrak{m}},\bar{\Lief},\bar\mu), (\bar{\mathfrak{m}},\bar{\Lief},\bar\mu)]=0$. Therefore, we have $[(\Lieh_1,\Liep_1,\sigma_1), (\Lieh_1,\Liep_1,\sigma_1)]$ $\cong [(\bar{\mathfrak{m}},\bar{\Lief},\bar\mu),(\bar{\mathfrak{m}},\bar{\Lief},\bar\mu)]$.
	
\noindent {\it (ii), (iii)}	Now, put $\Ker(\bar\beta)=(\mathfrak{\bar{t},\bar{s},\bar{\mu})=(t,s,\mu)/[(u,r,\mu),(m,f,\mu)]}$ and $Z(\mathfrak{\bar{m},\bar{f},\bar{\mu}})=(\mathfrak{\bar{k},\bar{l},\bar{\mu}})=\mathfrak{(k,l,\mu)/[(u,r,\mu),(m,f,\mu)]}.$  We claim that it is sufficient to prove that $Z(\mathfrak{m/t,f/s,\bar\mu}) =\mathfrak{(k/t,l/s/\bar\mu)}$. Hence we have
	$$\frac{\mathfrak{(h_1,p_1,\sigma_1)}}{Z(\mathfrak{h_1,p_1,\sigma_1})}\cong \frac{{(\mathfrak{\bar{m},\bar{f},\bar{\mu}})}/{\Ker(\bar{\beta})}}{\mathfrak{(k/t,l/s/\bar\mu)}} \cong\left( \mathfrak{\frac{m}{k},\frac{f}{l},\bar\mu} \right),$$
 $$\frac{Z(\Lieh_1,\Liep_1,\sigma_1)}{(\Liea_1,\Lieb_1,\sigma_1)} \cong \frac{\mathfrak{(k/t,l/s/\bar\mu)}}{(\bar{\Lieu}, \bar{\Lier},\bar{\mu})/(\Ker(\bar{\beta})} \cong \left( \mathfrak{\frac{k}{u}},\frac{l}{r},\bar\mu \right),$$
	and we get exactly what we want to prove.
	
	To prove our assertion, clearly, $\Ker(\bar{\beta})\subseteq  ({\bar{\Lieu},\bar{\Lier},\bar{\mu}}) \subseteq Z(\mathfrak{\bar{m},\bar{f},\bar{\mu}})$ and $\mathfrak{[(k,l,\mu)},$ $(\Liem,\Lief,\mu)] \subseteq (\Liet,\Lies,\mu)$. So, $\mathfrak{\left(k/t,l/s,\mu \right)}\subseteq Z\left( \mathfrak{m/t,f/s,\bar\mu} \right)$. To prove the opposite content, assume $Z\left( \mathfrak{m/t,f/s,\bar\mu} \right)=\left( \mathfrak{x/t,z/s,\bar\mu} \right)$, then by the assumptions we have
	$$\mathfrak{[(x,z,\mu),(m,f,\mu)]\subseteq (t,s,\mu)\cap [(m,f,\mu),(m,f,\mu)]=[(u,r,\mu),(m,f,\mu)]}.$$
	 Therefore,
	 	$$\mathfrak{\frac{(x,z,\mu)}{[(u,r,\mu),(m,f,\mu)]} \subseteq}  Z\left( \mathfrak{\frac{(m,f,\mu)}{[(u,r,\mu),(m,f,\mu)]}} \right) =  \mathfrak{\frac{(k,l,\mu)}{[(u,r,\mu),(m,f,\mu)]}}$$
 and so $Z\left( \mathfrak{m/t,f/s,\bar\mu} \right) =\mathfrak{\left(k/t,l/s/\bar\mu \right)}$. The proof is complete.	
\end{proof}
\begin{remark} \label{rem4.8}
	The above Corollary \ref{iso} shows that every perfect crossed module admits only one isomorphism class of stem covers.
\end{remark}

\begin{theorem}
  The central extension $(e) : 0 \to \Ker(\phi) \to (\Lieq \curlywedge \Lien,\Lieq \curlywedge \Lieq, \delta \curlywedge \Id) \stackrel{\phi} \to (\Lien,\Lieq,\delta) \to 0$ is a stem cover of $(\Lien,\Lieq,\delta)$ if and only if $(\Lien,\Lieq,\delta)$ is perfect.
\end{theorem}
\begin{proof}
  Let the extension $(e)$ be a stem cover of $(\Lien,\Lieq,\delta)$. Then we have $\im(\phi)=[(\Lien,\Lieq,\delta),(\Lien,\Lieq,\delta)]=(\Lien,\Lieq,\delta)$, so the crossed module $(\Lien,\Lieq,\delta)$ is perfect.

  Conversely, if $(\Lien,\Lieq,\delta)$ is perfect then $[\Lieq,\Lieq]=\Lieq$ and $D_{\Lieq}({\Lien})=\Lien$. For every $q\in \Lieq, n\in \Lien$, we can assume  $q=[q_1,q_2]$ and $n={^{q'}}{n'}$ for some $q_1,q_2, q' \in \Lieq$ and $n'\in \Lien$. Then
	$$q\curlywedge n=[q_1,q_2]\curlywedge {^{q'}}{n'}=[[q_1,q_2],q']\curlywedge n'- ^{[q_1,q_2]}{n'}\curlywedge q'=^{(q_1\curlywedge q_2)}q'\curlywedge n'\in D_{\Lieq \curlywedge \Lieq}{(\Lieq \curlywedge \Lien)},$$
	by the relations (4a) and (3c) in Definition \ref{non abelian}  and the Leibniz action of $\Lieq \curlywedge \Lieq$ on $\Lieq \curlywedge \Lien$. Consequently, $\Lieq \curlywedge \Lien \subseteq D_{\Lieq \curlywedge \Lieq}{(\Lieq \curlywedge \Lien)}$.

 Easily, $[\Lieq \curlywedge \Lieq,\Lieq \curlywedge \Lieq]=\Lieq \curlywedge \Lieq$, so the Leibniz  crossed module $(\Lien \curlywedge \Lieq,\Lieq \curlywedge \Lieq, \delta \curlywedge \Id)$ is perfect, now the result follows by  Corollary \ref{cor3.1} and Proposition \ref{central} {\it (iii)}.
\end{proof}

\begin{theorem}\label{yamazaki}
  Let $(\Lien,\Lieq,\delta)$ be a Leibniz crossed module such that  ${\cal M}(\Lien,\Lieq,\delta)$ is finite dimensional. If $(e): 0 \longrightarrow \mathfrak{(a,b,\sigma) \longrightarrow (h,p,\sigma) \stackrel{\varphi=(\varphi_1,\varphi_2)} \longrightarrow (n,q,\delta)} \longrightarrow 0$ is a stem extension of $(\Lien,\Lieq,\delta)$,
  then there is a stem cover $(e_1): 0 \longrightarrow \mathfrak{(a_{\rm 1},b_{\rm 1},\sigma_{\rm 1}) \longrightarrow}$ $\mathfrak{ (h_{\rm 1},p_{\rm 1},\sigma_{\rm 1}) \longrightarrow (n,q,\delta)} \longrightarrow 0 $ such that $(e)$ is homomorphic image of $(e_1)$.
\end{theorem}
\begin{proof}
Assume that $0 \longrightarrow (\mathfrak{u},\Lier,\mu) \longrightarrow (\mathfrak{m},\Lief,\mu) \stackrel{\pi=(\pi_1, \pi_2)} \longrightarrow (\Lien,\Lieq,\delta) \longrightarrow 0$ is a projective presentation of $(\Lien,\Lieq,\delta)$. Thanks to Proposition \ref{prop 4.2} {\it (ii)}, there is a surjective homomorphism $\overline{\beta}: (\bar{\mathfrak{m}},\bar{\Lief},\bar\mu)\longrightarrow (\Lieh,\Liep,\sigma)$ such that $\gamma \circ \bar{\beta}= \bar{\pi}$ and $\bar{\beta}\mathfrak{(\bar{u},\bar{r},\bar{\mu})}=(\Liea,\Lieb,\sigma)$. Putting $\Ker(\bar{\beta})=(\bar{\Liet},\bar{\Lies},\bar{\mu})$, then we have
$$\frac{(\mathfrak{u},\Lier,\mu)}{(\Liet,\Lies,\mu)}\cong (\Liea,\Lieb,\sigma)\cong \frac{((\mathfrak{u},\Lier,\mu)\cap [(\mathfrak{m},\Lief,\mu), (\mathfrak{m},\Lief,\mu)])+(\Liet,\Lies,\mu)}{(\Liet,\Lies,\mu)},$$
so $(\mathfrak{u},\Lier,\mu)=(\mathfrak{u},\Lier,\mu) \cap [(\mathfrak{m},\Lief,\mu),(\mathfrak{m},\Lief,\mu)]+ (\Liet,\Lies,\mu)$, because
$(\Liea,\Lieb,\sigma)$ has finite dimension. On the other hand, the following exact sequence splits by Lemma \ref{projective} {\it (ii)}:
$$0 \to [(\bar{\mathfrak{m}},\bar{\Lief},\bar{\mu}),(\bar{\mathfrak{m}},\bar{\Lief},\bar{\mu})] \cap (\bar{\Liet},\bar{\Lies},\bar{\mu}) \to (\bar{\Liet},\bar{\Lies},\bar{\mu}) \to
\frac{(\bar{\Liet},\bar{\Lies},\bar{\mu})}{[(\bar{\mathfrak{m}},\bar{\Lief},\bar{\mu}),(\bar{\mathfrak{m}},\bar{\Lief},\bar{\mu})]\cap (\bar{\Liet},\bar{\Lies},\bar{\mu})} \to 0$$
Thus $(\bar{\Liet},\bar{\Lies},\bar{\mu})= ([(\bar{\mathfrak{m}},\bar{\Lief},\bar{\mu}), (\bar{\mathfrak{m}},\bar{\Lief},\bar{\mu})]
\cap (\bar{\Liet},\bar{\Lies},\bar{\mu})) \oplus (\bar{\Liet_1},\bar{\Lies_1},\bar{\mu})$, for some  crossed  ideal $(\bar{\Liet_1},\bar{\Lies_1},\bar{\mu})$ of
 $(\bar{\mathfrak{m}},\bar{\Lief},\bar\mu)$, where $(\bar{\Liet_1},\bar{\Lies_1},\bar{\mu}) = \frac{(\bar{\Liet},\bar{\Lies},\bar{\mu})}{[(\bar{\mathfrak{m}},\bar{\Lief},\bar{\mu}),(\bar{\mathfrak{m}},\bar{\Lief},\bar{\mu})] \cap (\bar{\Liet},\bar{\Lies},\bar{\mu})}$. It yields that
 $\mathfrak{(\bar{u},\bar{r},\bar{\mu})}={\cal M}(\Lien,\Lieq,\delta) \oplus (\bar{\Liet_1},\bar{\Lies_1},\bar{\mu})$ and so by Theorem \ref{Th 4.1} {\it (i)}, the extension $(e_1): 0 \to \mathfrak{(u/t_{\rm 1},r/s_{\rm 1},\bar{\mu})\to (m/t_{\rm 1},f/s_{\rm 1},\bar{\mu})}$
 $\mathfrak{\to (n,q,\delta)} \to 0$ is a stem cover of $\mathfrak{(n,q,\delta)}$
  and the extension $(e)$ is homomorphic image of $(e_1)$, as required.
\end{proof}

An immediate consequence of the above theorem gives  a new characterization of stem covers of finite-dimensional Leibniz crossed modules.

\begin{corollary} \label{stem iso}
A stem extension $(e)$ of a finite dimensional crossed module $(\Lien,\Lieq,\delta)$ is a stem cover if and only if any surjective homomorphism of other stem extension of $(\Lien,\Lieq,\delta)$ onto $(e)$ is an isomorphism.
\end{corollary}
\begin{proof}
	The sufficient condition follows from Theorem \ref{yamazaki}.
	 For necessary condition, 	let $(e): 0 \longrightarrow  \mathfrak{(a,b,\sigma) \longrightarrow  (h,p,\sigma) \longrightarrow  (n,q,\delta)} \longrightarrow  0$ be a stem cover of $(\Lien,\Lieq,\delta)$ and $(e_1)$ be a stem extension of $(\Lien,\Lieq.\delta)$ such that there is a surjective homomorphism $\alpha = (\alpha_1, \alpha_2) : (e_1) \longrightarrow (e)$. According to Theorem \ref{yamazaki}, we can find a stem cover $(e_2): 0 \longrightarrow  \mathfrak{(a_2,b_2,\sigma_2) \longrightarrow  (h_2,p_2,\sigma_2) \longrightarrow  (n,q,\delta)} \longrightarrow  0$ of $(\Lien,\Lieq.\delta)$ such that $(e_1)$ is homomorphic image of $(e_2)$ and $\beta = (\beta_1, \beta_2) : (e_2) \longrightarrow (e_1)$ is a surjective homomorphism. So, we have the surjective homomorphism $\alpha \circ \beta$ from $(e_2)$ onto $(e)$. Now, since $(\Lien,\Lieq,\delta)$ is finite dimensional, then we have
 $({\rm dim}(\Lieh_2), {\rm dim}(\Liep_2)) = ({\rm dim}(\Lien), {\rm dim}(\Lieq))+ ({\rm dim}(\Liea_2), {\rm dim}(\Lieb_2)) = ({\rm dim}(\Lien),{\rm dim}(\Lieq))+ ({\rm dim}(\Liea), {\rm dim}(\Lieb)) = ({\rm dim}(\Lieh), {\rm dim}(\Liep))$.
  Therefore, $\alpha \circ \beta$ is an isomorphism and then $\alpha$ is an isomorphism, as required.
\end{proof}


 \section{Connection with stem cover of Lie crossed modules} \label{connection Lie}
 In this section, we investigate the interplay between the notions of stem cover and the Schur multiplier of Leibniz crossed modules with the similar of these concepts for Lie crossed modules in \cite{CL}.

 \begin{theorem} \label{Th 5.1}
Let $(T,L,\tau)$ be a crossed module of Lie algebras  with finite dimensional Schur multiplier ${\cal M}(T,L,\tau)$ as Leibniz crossed module and let $(e): 0 \longrightarrow (A,B,\vartheta) \longrightarrow (M,P,\vartheta) \longrightarrow (T,L,\tau) \longrightarrow 0$ be a stem cover of $(T,L,\tau)$ in {\bf XLie}. Then there exists a stem cover $(e_1): 0 \longrightarrow \mathfrak{(a,b,\sigma)\longrightarrow  (h,p,\sigma) \longrightarrow} (T,L,\tau) \longrightarrow 0$ of Leibniz crossed modules, such that $(e)$ is homomorphic image of $(e_1)$.

Moreover, if $(T,L,\tau)$ is finite dimensional Lie crossed module then $\mathfrak{(h,p,\sigma)}_{\Lie}\cong (M,P,\vartheta)$.
\end{theorem}
\begin{proof}
  Clearly, $(e)$ is a stem extension of $(T,L,\tau)$ in $\mathbf{XLb}$.
   According to Theorem \ref{yamazaki}, there exists a stem cover $(e_1): 0 \to \mathfrak{(a,b,\sigma) \to  (h,p,\sigma) \to} (T,L,\tau) \to 0$ and a surjective homomorphism $\beta=(\beta_1,\beta_2):(\Lieh,\Liep,\sigma) \twoheadrightarrow (M,P,\vartheta)$.
    Obviously, $\beta$ induces a surjective homomorphism $\bar{\beta}=(\bar{\beta}_1,\bar{\beta}_2):(\Lieh,\Liep,\sigma)_{\Lie}\longrightarrow (M,P,\vartheta)$.

    On the other hand, $0 \longrightarrow (\mathfrak{(a,b,\sigma)+(h,p,\sigma)}^{\rm ann})/(\Lieh,\Liep,\sigma)^{\rm ann} \longrightarrow (\Lieh,\Liep,\sigma)_{\Lie} \longrightarrow (T,L,\tau) \longrightarrow 0$ is a stem extension of $(T,L,\tau)$ in the  category  $\mathbf{XLie}$ of crossed modules in Lie algebras,
so by the proof of Theorem 3.6 in \cite{RS}, there is a surjective homomorphism $\alpha=(\alpha_1, \alpha_2)$ from $(M_1,P_1,\vartheta_1)$ to $(\Lieh,\Liep,\sigma)_{\Lie}$, where $(M_1,P_1,\vartheta_1)$ is a stem cover of $(T,L,\tau)$ in $\mathbf{XLie}$.
Now, combining \cite[Corollary 3.5]{RS} with \cite[Proposition 22]{OUI} and \cite[Theorem 13]{Mon}, we deduce that $M \cong M_1$ and $P \cong P_1$. So, the homomorphisms $M \longrightarrow M_1 \stackrel{\alpha_1} \longrightarrow \bar{\Lieh} \stackrel{\bar{\beta_1}} \longrightarrow M$ and $P \longrightarrow P_1 \stackrel{\alpha_2} \longrightarrow \Liep_{\Lie}\stackrel{\bar{\beta_2}} \longrightarrow P$ are surjective. Thus, they are isomorphism, in finite dimensional case.

It is easy to check that $\bar{\beta_1}$ and $\bar{\beta_2}$ are isomorphisms and so $\bar{\beta}$ is an isomorphism of crossed modules, as required.
\end{proof}

\begin{remark}
Theorem \ref{Th 5.1} applied to the particular Lie crossed modules $(T, T, \Id)$ or $(0, T, inc)$ recovers the corresponding results for Lie algebras provided in \cite[Theorem 3.4]{EV}.
\end{remark}

Note that in proof of the above Theorem \ref{Th 5.1} we have $\Ker(\beta) \subseteq (\Liea,\Lieb,\sigma)$.
 This fact and exact sequence (\ref{five term}) provide the following consequence:

 \begin{corollary}
 Under the assumptions of Theorem \ref{Th 5.1},  the exact sequence $(\hat{e}): 0 \to \Ker(\beta) \to (\Lieh,\Liep,\sigma) \to (M,P,\vartheta) \to 0$ is a stem extension of $(M,P,\vartheta)$ in the category  ${\bf XLb}$.

  Moreover, if $(T,L,\tau)$ is a  perfect Lie crossed module, then $(\hat{e})$ is a stem cover of $(M,P,\vartheta)$ in ${\bf XLb}$.
\end{corollary}
\begin{proof}
 By the assumptions of Theorem \ref{Th 5.1} we have $\Ker(\beta) \subseteq (\Liea,\Lieb,\sigma)\subseteq Z(\Lieh,\Liep,\sigma)$ $\cap [(\Lieh,\Liep,\sigma),(\Lieh,\Liep,\sigma)]$. So  $(\hat{e})$ is stem extension. By application of  sequence (\ref{five term}), we obtain the following exact sequence:
	$$\mathcal{M}(\Lieh,\Liep,\sigma)\longrightarrow \mathcal{M}(M,P,\vartheta) \longrightarrow \Ker(\beta) \longrightarrow (\Lieh,\Liep,\sigma)_{\rm ab} \longrightarrow (M,P,\vartheta)_{\rm ab} \longrightarrow 0$$
Now if $(T,L,\tau)$ is a perfect Lie crossed module,  it also is a perfect Leibniz crossed module. Since $(\hat{e})$ is stem cover $(T,L,\tau)$, then thanks to Corollary \ref{cor4.1} we have $(\Lieh,\Liep,\sigma)_{\rm ab}= \mathcal{M}(\Lieh,\Liep,\sigma)=0$. Therefore, $\Ker(\beta) \cong \mathcal{(M,P,\vartheta)}$. The proof is complete.
\end{proof}

\begin{corollary} \label{cor5.4}
Let $(T,L,\tau)$ be a Lie crossed module.  Then ${\cal M}^{\Lie} (T,L,\tau)$ (the Schur multiplier in {\bf XLie} \cite{CL}) is homomorphic image of ${\cal M}(T,L,\tau)$.
\end{corollary}

\begin{corollary}
	Let $(T,L,\tau)$ be a perfect Lie crossed module  with finite dimensional Schur multiplier ${\cal M}(T,L,\tau)$ as a Leibniz crossed module. Then the stem cover $(e):  0 \longrightarrow \mathfrak{(a,b,\sigma)\longrightarrow  (h,p,\sigma) \stackrel{\varphi} \longrightarrow} (T,L,\tau) \longrightarrow 0$ of $(T,L,\tau)$ in $\mathbf{XLb}$ is a stem cover of $(T,L,\tau)$ in $\mathbf{XLie}$, if and only if  $\mathcal{M}(T,L,\tau)\cong \mathcal{M}^{Lie}(T,L,\tau)$.
\end{corollary}
\begin{proof}
	Let $(e_1):  0 \longrightarrow (A,B,\sigma_1)\longrightarrow  (H,P,\sigma_1) \stackrel{\bar\varphi}\longrightarrow (T,L,\tau) \longrightarrow 0$ be a stem cover of $(T,L,\tau)$ in $\mathbf{XLie}$. Then $(e_1)$ is a stem extension of $(T,L,\tau)$ in $\mathbf{XLb}$, so by Theorem \ref{yamazaki} and Remark \ref{rem4.8}, there is a surjective homomorphism $\beta: (\Lieh,\Liep,\sigma)\longrightarrow (H,P,\sigma_1)$ such that the following diagram is commutative:
	\[ \xymatrix{
		0 \ar[r] & (\Liea,\Lieb,\sigma) \ar[r] \ar[d]^{\beta_{\mid}}& (\Lieh,\Liep,\sigma) \ar[r]^{{\varphi}} \ar[d]^{\beta} & (T,L,\tau) \ar[r] \ar@{=}[d] & 0\\
		0 \ar[r] & (A,B,\sigma_1) \ar[r] &(H,P,\sigma_1) \ar[r]^{\bar\varphi} & (T,L,\tau) \ar[r] & 0.
	} \]
	
	Note that, $\Ker(\beta) \subseteq \mathfrak{(a,b,\sigma)}$, so the  restriction  $\beta_{\mid}$ from $(\Liea,\Lieb,\sigma)$ onto $(A,B,\sigma_1)$ is an isomorphism if and only if $\beta$ is an isomorphism, Thus, by finiteness of $\mathcal{M}(\Liea,\Lieb,\sigma)$, we conclude $\mathcal{M}(T,L,\tau)\cong \mathcal{M}^{Lie}(T,L,\tau)$ if and only if $(H,P,\sigma_1) \cong (\Lieh,\Liep,\sigma)$. The result follows.
\end{proof}

Corollary \ref{cor5.4} is extended to  Leibniz crossed modules as follows:

\begin{theorem}
  Let $(\Lien,\Lieq,\delta)$ be a Leibniz crossed module.  Then ${\cal M}^{\Lie}((\Lien,\Lieq,\delta)_{\Lie})$ is homomorphic image of ${\cal M}(\Lien,\Lieq,\delta)$.
\end{theorem}
\begin{proof}
  Let $0 \longrightarrow  (\mathfrak{u},\Lier,\mu) \longrightarrow (\mathfrak{m},\Lief,\mu) \stackrel{\pi}\longrightarrow (\Lien,\Lieq,\delta) \longrightarrow 0$ be a projective presentation of $(\Lien,\Lieq,\delta)$. Since, $\mathfrak{(m,f,\mu)}_{\Lie}$ is a projective crossed module in  $\mathbf{XLie}$, then we have the projective presentation
  $$0 \to \frac{\mathfrak{(u,r,\mu)}+\mathfrak{(m, f,\mu)}^{\rm ann}}{(\mathfrak{m},\Lief,\mu)^{\rm ann}}\to (\mathfrak{m},\Lief,\mu)_{\Lie}\to (\Lien,\Lieq,\delta)_{\Lie} \to 0$$
  of the Lie crossed module $(\Lien,\Lieq,\delta)_{\Lie}$. Clearly, the projection homomorphism $pr:(\mathfrak{m},\Lief,\mu)\to (\mathfrak{m},\Lief,\mu)_{\Lie}$ induces the surjective homomorphism $\bar{pr}:[(\mathfrak{m},\Lief,\mu),(\mathfrak{m},\Lief,\mu)]$ $\to [(\mathfrak{m},\Lief,\mu)_{\Lie},(\mathfrak{m},\Lief,\mu)_{\Lie}]$, which gives rise to a surjective homomorphism
 \[
\xymatrix{
{\cal M}(\Lien,\Lieq,\delta )=\frac{(\mathfrak{u},\Lier,\mu)\cap [(\mathfrak{m},\Lief,\mu),(\mathfrak{m},\Lief,\mu)]}{[(\mathfrak{u},\Lier,\mu),(\mathfrak{m},\Lief,\mu)]}
 \ar@{>>}[d]_{\tilde{pr}} &\\
 {\cal M}^{\Lie}((\Lien,\Lieq,\delta)_{\Lie}) = \frac{\overline{(\mathfrak{u},\Lier,\mu)}\cap [(\mathfrak{m},\Lief,\mu)_{\Lie},(\mathfrak{m},\Lief,\mu)_{\Lie}]}{[\overline{(\mathfrak{u},\Lier,\mu)},
  (\mathfrak{m},\Lief,\mu)_{\Lie}]},
  }
  \]
  where $\overline{(\mathfrak{u},\Lier,\mu)}=(\mathfrak{(u,r,\mu)}+\mathfrak{(m,f,\mu)}^{\rm ann})/(\mathfrak{m},\Lief,\mu)^{\rm ann}$, and the result follows.
\end{proof}



\section*{Acknowledgements}

First author was supported by  Agencia Estatal de Investigación (Spain), grant MTM2016-79661-P (AEI/FEDER, UE, support included).


\begin{center}

\end{center}

\end{document}